\let\oldvec\vec% Store \vec in \oldvec
\RequirePackage{fix-cm}

\documentclass[smallextended]{svjour3}       % onecolumn (second format)
\let\vec\oldvec% Restore \vec from \oldvec

\smartqed  % flush right qed marks, e.g. at end of proof
\usepackage{graphicx}
\usepackage{amssymb}
\usepackage{amsmath}
\usepackage{epsfig}
\makeatletter
\newenvironment{proofof}[1]{\par
  \pushQED{\qed}%
  \normalfont \topsep6\p@\@plus6\p@\relax
  \trivlist
  \item[\hskip\labelsep
        \itshape
    Proof of #1\@addpunct{.}]\ignorespaces
}{%
  \popQED\endtrivlist\@endpefalse
}
\makeatother
\usepackage{amsthm}
\begin{document}

\title{General Multi-sum Transformations and Some Implications
}
%\subtitle{Do you have a subtitle?\\ If so, write it here}

%\titlerunning{Short form of title}        % if too long for running head

\author{James Mc Laughlin
}

%\authorrunning{Short form of author list} % if too long for running head

\institute{James Mc Laughlin \at
              Mathematics Department,
 25 University Avenue,
West Chester University, West Chester, PA 19383 \\
              Tel.: + 	610-430-4417\\
              Fax: + 610-738-0578\\
              \email{jmclaughl@wcupa.edu}           %  \\
%             \emph{Present address:} of F. Author  %  if needed
 %          \and
 %          S. Author \at
 %             second address
}

\date{Received: \today / Accepted: date}
% The correct dates will be entered by the editor

\maketitle

\begin{abstract}
We give two general transformations that allows certain quite general basic hypergeometric multi-sums of arbitrary depth (sums that involve an arbitrary sequence $\{g(k)\}$), to be reduced to an infinite $q$-product times a single basic hypergeometric sum.  Various applications are given, including summation formulae for some $q$ orthogonal polynomials, and various multi-sums that are expressible as infinite products.

% \PACS{PACS code1 \and PACS code2 \and more}
\end{abstract}

\section{Introduction}\label{intro}
The main results of the paper are two  general multi-sum-to-single-sum transformations, in which (assuming convergence on each side) an arbitrary sequence $\{g_k\}$ may be input on each side of the transformations (see Theorems \ref{multsumt} and \ref{multsumtt} below).

Such transformations are of course not new, and indeed the iteration of any Bailey- or WP-Bailey chain (see, for example, \cite{A01,W03,LM09,MZ10}) will produce such a transformation, if the ``$\beta_n$" on the multi-sum side are replaced with their defining sums over the ``$\alpha_j$", so that both sides become sums over a single sequence $\{\alpha_j\}$.  When the $\{\alpha_j\}$ is suitably chosen, the single series side may be summed as an infinite product.

Andrews in \cite{A84}, for example, showed how each of Slater's 130  identities in \cite{S52} may be embedded in an infinite family of multi-sum identities. One example of such a family of identities are the (analytic version of) the Andrews-Gordon identities (the case $k=2$ gives the Rogers-Ramanujan identities).
\begin{theorem}
For integers $k\geq 2$ and $1\leq i \leq k$,
\begin{multline}\label{AGids}
\sum_{m_1\geq m_2\geq \dots \geq m_{k-1}\geq 0}\frac{q^{m_1^2+m_2^2+\dots+m_{k-1}^2+m_i+m_{i+1}+\dots+m_{k-1}}}
{(q;q)_{m_1-m_{2}}  \dots (q;q)_{m_{k-2}-m_{k-1}}(q;q)_{m_{k-1}} }\\
=\frac{(q^i,q^{2k+1-i},q^{2k+1};q^{2k+1})_{\infty}}{(q;q)_{\infty}}.
\end{multline}
\end{theorem}
This result was first proved by Andrews \cite{A74}, but our statement of it is based on Chapman's \cite{C05} version, since the notation he uses is closer to that used in the present paper. Before coming to the identities in the present paper, we briefly consider some other multi-sum transformations in the literature.

Multi-sum identities were further considered in \cite{A81} by Andrews. However, the identities in that paper coincide with those in the present paper only in certain cases, and in these cases only when the depth of the multi-sum is either one or two. For example, Andrews proves a multi-sum generalization
of Cauchy's identity
\begin{equation}\label{ceq}
\sum_{n=0}^{\infty}\frac{q^{n^2}z^n}{(q,zq;q)_n} = \frac{1}{(zq;q)_{\infty}}
\end{equation}
of the form  (\cite[page 12, eq (1.5)]{A81})
\begin{equation}\label{aceq}
\sum_{n_1,n_2,\dots n_{k-1}=0}^{\infty}\frac{q^{N_1^2+N_2^2+\dots N_{k-1}^2}z^{N_1+N_2+\dots N_{k-1}}}
{(q;q)_{n_1}(q;q)_{n_2}\dots (q;q)_{n_{k-1}}(zq;q)_{n_{k-1}}} = \frac{1}{(zq;q)_{\infty}},
\end{equation}
where $N_i=n_i+n_{i+1}+\dots + n_{k-1}$. It can be seen that the $k=2$ case of \eqref{aceq} and the $k=1$ case of \eqref{mseq1}  (after setting $c_1=zq$ and $g(j)=\delta_{0,j}$) both reduce to \eqref{ceq}, but that quite different identities are given for larger $k$ (no matter how the parameters in \eqref{mseq1} are specialized). As another illustration of the differences between the general identities in the two papers, Andrews gives another proof
(\cite[page 16, Corollary 1]{A81}) of \eqref{AGids}, and the right side of this identity coincides with the right side of \eqref{mseq103} when $p=2k+1$, but clearly the left sides are very different. A third difference is that Lemmas 1 and 2 in \cite{A81} are not derivable from the identities in the present paper, primarily for the reason that if the summation indices in Theorems \ref{multsumt} and \ref{multsumtt} are redefined so that they all start at 0 (instead of being nested), the general terms in the multi-sums  contain terms of the form $(x;q)_{N_i}$, rather than $(x;q)_{n_i}$, where we are using the notation defined at \eqref{aceq}. Our identity at \eqref{mnkqid} was also derived by Andrews in \cite{A81} (and also previously by Andrews in \cite{A77}). Andrews \cite[page 18, Eq. (4.3)]{A81} also proved the identity
{\allowdisplaybreaks
\begin{equation}\label{mnk2qid}
\sum_{m, n,r\geq
0}\frac{q^{km^2-(2k-2)mn+kn^2+r^2+mr+nr}}{(q;q)_m(q;q)_n(q;q)_r}=\frac{(-q^k,-q^k,q^{2k};q^{2k})_{\infty}}{(q;q)_{\infty}}.
\end{equation}
}
As another illustration of how the transformations in the present paper diverge from those in the paper of Andrews \cite{A81}  at greater depth (number of summation variables), the corresponding (depth three) identity in the present paper is
\begin{equation}\label{mnktwoqid}
\sum_{m, n,r\geq
0}\frac{q^{km^2-(2k-1)mn+kn^2+r^2+mr}}{(q;q)_{m+r}(q;q)_m(q;q)_n(q;q)_r}=\frac{(-q^k,-q^k,q^{2k};q^{2k})_{\infty}}{(q;q)_{\infty}^2},
\end{equation}
which follows from \eqref{mseq111} upon setting $k=2$, $c_1=q$, $g(j)=q^{kj^2}$ and finally re-indexing the summation variables so that they all start at 0.

The main  identity Chu's 2002 paper \cite[page 581, Lemma 1]{Ch02} derives from the $q$-Pfaff-Saalsch\"utz sum (see \cite[page 355, Eq. (II.12)]{GR04}) and may be expressed as
{\allowdisplaybreaks
\begin{multline}\label{chu02eq}
\prod_{i=1}^{n}\frac{(x_i,y_i;q)_k}{(qa/x_i,qa/y_i;q)_k}
\left(\frac{qa}{x_iy_i}  \right)^k=
\left( \frac{a}{c}\right)^N
\prod_{i=1}^{n}
\frac{(qc/x_{i},qc/y_{i};q)_{N_i}}{(qa/x_i,qa/y_i;q)_{N_i}}\\
\sum_{\tilde{m}\geq 0}
q^{M_n}\prod_{i=1}^{n-1}
\frac{(qa/x_iy_i;q)_{m_i}}{(q;q)_{m_i}}
\frac{(x_{i+1}c/a,y_{i+1}c/a;q)_{M_i}}{(qc/x_i,qc/y_i;q)_{M_i}}
\left(\frac{qa}{x_{i+1}y_{i+1}}  \right)^{M_i}\\
\frac{(qa/x_ny_n;q)_{m_n}}{(q;q)_{m_n}}
\frac{(c,c/a;q)_{M_n}}{(qc/x_n,qc/y_n;q)_{M_n}}
\frac{(c q^{M_n},qa/c;q)_{k}}{(c,q^{1-M_n}a/c;q)_{k}}q^{-kM_n},
\end{multline}
}
where $k$ is a non-negative integer,
\[
M_i=\sum_{j=1}^im_j, \hspace{25pt} q^{1+N_k}=x_ky_k, \hspace{25pt}  N=\sum_{i=1}^nN_i,
\]
and the multiple summation index $\tilde{m}=(m_1,m_2,\dots m_n)$ runs over all $m_i\geq 0$ for $i=1,2,\dots , n$. When $c=a$, this identity simplifies to the main
 identity in  Chu's 2005 paper \cite[page 103, Lemma 2.1]{Ch05},
 \begin{multline}\label{chu05eq}
\prod_{i=1}^{n}\frac{(x_i,y_i;q)_k}{(qa/x_i,qa/y_i;q)_k}
\left(\frac{qa}{x_iy_i}  \right)^k=
\sum_{\tilde{m}\geq 0}
\frac{(qa/x_ny_n;q)_{m_n}}{(q;q)_{m_n}}
\frac{(q^{-k},q^ka;q)_{M_n}}{(qa/x_n,qa/y_n;q)_{M_n}}\\
q^{M_n}\prod_{i=1}^{n-1}
\frac{(qa/x_iy_i;q)_{m_i}}{(q;q)_{m_i}}
\frac{(x_{i+1},y_{i+1};q)_{M_i}}{(qa/x_i,qa/y_i;q)_{M_i}}
\left(\frac{qa}{x_{i+1}y_{i+1}}  \right)^{M_i}.
\end{multline}
  This latter identity was also derived by Andrews \cite[page 19, Eq (5.2)]{A79}, and was the key identity used by him in sections 5 and 6 of that paper, the sections dealing with multi-sums. Several general transformations are subsequently derived by multiplying each side of either \eqref{chu02eq} or \eqref{chu05eq} by $W_k$, where $\{W_k\}$ is an arbitrary sequence, and particular identities are derived by specializing the sequence $W_k$. It is possible to make some comparisons  between the transformations in the present paper and those in the papers of Andrews \cite{A79} and Chu \cite{Ch02,Ch05}, by comparing the identities at \eqref{chu02eq} and  \eqref{chu05eq} with the identity at \eqref{multc}    with $g(j)=\delta_{0,j}$ (the identity at \eqref{multcc} with $g(j)=\delta_{0,j}$
  reduces to a special case of  \eqref{multc}    with $g(j)=\delta_{0,j}$). The most obvious difference is that the summation formulae of Andrews and Chu being considered involve finite sums and finite products, while that at \eqref{multc} involves infinite sums and infinite products. Of course it is a simple matter to convert the infinite product on the right side of \eqref{multc} into a finite product by setting each $a_j=q^{-n_j}$ for positive integers $n_j$. While considering this, we observed a somewhat curious phenomenon - while the right side of \eqref{multc} becomes a finite product, the left side does not necessarily become a finite multi-sum (we are setting $g(j)=\delta_{0,j}$ in \eqref{multc}), as indicated in the following Corollary to Theorem \ref{multsumt}.
  \begin{corollary}
  Let $n_1$, $n_2$ $\dots$, $n_k$ be positive integers, and $b_1, \dots , b_k$,  $c_1, \dots , c_k$ be complex numbers.  Then
\begin{multline}\label{multcfin}
\sum_{\vec{m}}\prod_{j=1}^{k-1}
\frac{(q^{m_{j+1}-n_j},b_j;q)_{m_j-m_{j+1}}(c_j/b_j;q)_{m_{j+1}}}
{(c_j;q)_{m_j}(q;q)_{m_j-m_{j+1}}}\left(\frac{c_jq^{n_j}}{b_j}\right)^{m_j-m_{j+1}}\\
\times
\frac{(q^{-n_k},b_k;q)_{m_k}}
{(c_k,q;q)_{m_k}}\left( \frac{c_kq^{n_k}}{b_k}\right)^{m_k}
=\prod_{j=1}^{k}\frac{(c_j/b_j;q)_{n_j}}{(c_j;q)_{n_j}},
\end{multline}
provided either the multi-sum terminates, or the values of the parameters are such that it converges if it does not terminate. The multi-sum terminates
if and only if $n_1\geq n_2 \geq \dots \geq n_k$.
\end{corollary}
Observe that the $q$-products on the right side of \eqref{multcfin} may be of different orders, in contrast to those on the left side of \eqref{chu02eq} or  \eqref{chu05eq}, which are all of order $k$.
As regards infinite identities, setting $g(j)=\delta_{0,j}$ in \eqref{multc},  replacing $c_j$ with $x_j$ and $b_j$ with $x_j/y_j$, and then re-indexing the summation variables so that each $m_j$ runs independently over the range $m_j \geq 0$,  gives rise to the summation formula (assuming the choice of parameters leads to convergence of the multi-sum)
\begin{multline}\label{multcinf}
\sum_{m_1,m_2,\dots,m_k\geq 0}\prod_{j=1}^{k-1}\frac{(a_j;q)_{M_j}(y_j;q)_{M_{j+1}}(x_j/y_j;q)_{m_j}}
{(x_j;q)_{M_j}(a_j;q)_{M_{j+1}}(q;q)_{m_j}}\left( \frac{y_j}{a_j}\right)^{m_j}\\
\times \frac{(a_k;q)_{m_k}(x_k/y_k;q)_{m_k}}
{(x_k;q)_{m_k}(q;q)_{m_k}}\left( \frac{y_k}{a_k}\right)^{m_k}
=\prod_{j=1}^{k}\frac{(x_j/a_j,y_j;q)_{\infty}}{(x_j,y_j/a_j;q)_{\infty}},
\end{multline}
where this time $M_i=\sum_{j=i}^km_j$. The special case derived by setting each $a_j=a$ also follows from \eqref{chu05eq} and was stated by Chu \cite[page 103, Corollary 2.2]{Ch05}:
{\allowdisplaybreaks
\begin{multline}\label{multcinf2}
\sum_{m_1,m_2,\dots,m_k\geq 0}(a;q)_{M_1}\prod_{j=1}^{k}\frac{(y_j;q)_{M_{j+1}}(x_j/y_j;q)_{m_j}}
{(x_j;q)_{M_j}(q;q)_{m_j}}\left( \frac{y_j}{a}\right)^{m_j}\\
=\prod_{j=1}^{k}\frac{(x_j/a,y_j;q)_{\infty}}{(x_j,y_j/a;q)_{\infty}}.
\end{multline}
}
The further specialization derived by letting each $y_j \to 0$ and $a \to \infty$, namely
\begin{multline}\label{multcinf3}
\sum_{m_1,m_2,\dots,m_k\geq 0}q^{M_1(M_1-1)/2}\prod_{j=1}^{k}\frac{x_j^{m_j}q^{m_j(m_j-1)/2}}
{(x_j;q)_{M_j}(q;q)_{m_j}}\left( \frac{y_j}{a}\right)^{m_j}\\
=\frac{1}{\prod_{j=1}^{k}(x_j;q)_{\infty}},
\end{multline}
was also stated by Chu \cite[Corollary 2.3]{Ch05}, Andrews \cite[Eq. (6.1)]{A79} and Milne \cite[Thm. 3.1]{M80}.
The two identities in Corollary \ref{multeuler} may be derived as special cases of the above identity.

Another group of multi-sum identities contains generalizations of classical single-sum transformation- and summation identities to multi-sum extensions.
See the papers by Gustafson \cite{G87},  Milne and Schlosser \cite{MS02},  Rosengren and Schlosser \cite{RS03}, and Spiridonov and Warnaar \cite{SW11}, and other papers listed in the bibliography of these papers, for some examples.

One of the two main result in the present paper is the the multi-sum transformation formula contained in the following theorem.

\begin{theorem}\label{multsumt}
Let $|q|<1$, $k\geq 1$ be a positive integer,  $a_1, \dots , a_k$,  $b_1, \dots , b_k$,  $c_1, \dots , c_k$ be complex numbers, and $\{g(j)\}_{j=0}^{\infty}$ be a sequence of numbers such that both series below converge. Let the sum on the left below be over all integer  $k+1$-tuples $\vec{m}=(m_1,m_2,\dots, m_{k+1})$ %satisfying
with $m_1\geq m_2\geq \dots \geq m_k\geq m_{k+1}\geq 0$. Then
\begin{multline}\label{multc}
\sum_{\vec{m}}\prod_{j=1}^{k}\frac{(a_j;q)_{m_j}(c_j/b_j;q)_{m_{j+1}}(b_j;q)_{m_j-m_{j+1}}}
{(c_j;q)_{m_j}(a_j;q)_{m_{j+1}}(q;q)_{m_j-m_{j+1}}}\left( \frac{c_j}{a_jb_j}\right)^{m_j-m_{j+1}} g(m_{k+1})\\
=\prod_{j=1}^{k}\frac{(c_j/a_j,c_j/b_j;q)_{\infty}}{(c_j,c_j/a_jb_j;q)_{\infty}}\sum_{j=0}^{\infty}g(j).
\end{multline}
\end{theorem}

Two special cases of this identity are contained in the following corollary.
\begin{corollary}\label{multeuler}
Let $k\geq 1$ be an integer, and let the sum on the left be over all integer  $k$-tuples $\vec{m}=(m_1,m_2,\dots, m_{k})$ satisfying $m_1\geq m_2\geq \dots \geq m_k\geq 0$. Then
\begin{equation}\label{mseq10}
\sum_{\vec{m}}\frac{q^{m_1(m_1-m_2)+m_2(m_2-m_3)+\dots +m_{k-1}(m_{k-1}-m_{k})+m_k^2}}
{(q;q)_{m_1}(q;q)_{m_1-m_{2}}  \dots (q;q)_{m_{k-1}}(q;q)_{m_{k-1}-m_{k}}(q;q)_{m_k}^2 }
=\frac{1}{(q;q)_{\infty}^k};
\end{equation}
\begin{multline}\label{mseq102}
\sum_{\vec{m}}\frac{q^{k\left[m_1(m_1-m_2)+m_2(m_2-m_3)+\dots +m_{k-1}(m_{k-1}-m_{k})+m_k^2-m_1\right]+m_1+m_2+\dots+m_k}}
{(q^k;q^k)_{m_k}\prod_{i=1}^{k-1}(q^k;q^k)_{m_i-m_{i+1}}\prod_{i=1}^k(q^i;q^k)_{m_i} }\\
=\frac{1}{(q;q)_{\infty}}.
\end{multline}
\end{corollary}
The  identity of Jacobi,
\[
\sum_{n=0}^{\infty}\frac{q^{n^2}}{(q;q)_n^2}=\frac{1}{(q;q)_{\infty}},
\]
may be viewed as the $k=1$ case of each of the identities at \eqref{mseq10} and \eqref{mseq102} above, so that each of  \eqref{mseq10} and \eqref{mseq102} embeds Jacobi's identity in an infinite family of identities.
Just as Jacobi's identity has a combinatorial interpretation (each side being the generating function for the number of unrestricted partitions of a positive integer), it maybe that \eqref{mseq10} has a combinatorial interpretation in terms of multipartitions with $k$ components. Similarly, the left side of
\eqref{mseq102} may have an interpretation in terms of $k$-modular partitions. We leave these questions as open problems for the reader.

A variation of Theorem \ref{multsumt} which results in a bilateral infinite series on the single-sum side is given by modifying the innermost sum on the multi-sum side.

\begin{theorem}\label{multsumtt}
Let $|q|<1$, $k\geq 1$ be a positive integer,  $a_1, \dots , a_{k-1}$,  $b_1, \dots , b_{k-1}$,  $c_1, \dots , c_{k-1}$, and $a$ be complex numbers, and $\{g(j)\}_{j=-\infty}^{\infty}$ be a sequence of numbers such that both series below converge. Let the sum on the left below be over all integer  $k+1$-tuples $\vec{m}=(m_1,m_2,\dots, m_{k+1})$ %satisfying
with $m_1\geq m_2\geq \dots \geq m_k\geq 0$, $m_k\geq  m_{k+1}>-\infty$. Then
{\allowdisplaybreaks
\begin{multline}\label{multcc}
\sum_{\vec{m}}\prod_{j=1}^{k-1}\frac{(a_j;q)_{m_j}(c_j/b_j;q)_{m_{j+1}}(b_j;q)_{m_j-m_{j+1}}}
{(c_j;q)_{m_j}(a_j;q)_{m_{j+1}}(q;q)_{m_j-m_{j+1}}}\left( \frac{c_j}{a_jb_j}\right)^{m_j-m_{j+1}}\\
\times\frac{(a;q)_{m_k}(q/a;q)_{m_{k+1}}(a;q)_{m_k-m_{k+1}}}
{(q;q)_{m_k}(a;q)_{m_{k+1}}(q;q)_{m_k-m_{k+1}}}\left( \frac{q}{a^2}\right)^{m_k-m_{k+1}} g(m_{k+1})\\
=\frac{(q/a,q/a;q)_{\infty}}{(q,q/a^2;q)_{\infty}}\prod_{j=1}^{k-1}\frac{(c_j/a_j,c_j/b_j;q)_{\infty}}{(c_j,c_j/a_jb_j;q)_{\infty}}
\sum_{j=-\infty}^{\infty}g(j).
\end{multline}
}
\end{theorem}
In the next identity, which is a special case of the above theorem, the right side coincides with the right side of the identity of Andrews at \eqref{AGids}, when $p$ is odd.
\begin{corollary}\label{corbilat}
Let $k\geq 1$, $p\geq 3$ and $i\leq p/2$  be positive integers, and let the sum on the left be over all integer  $k+1$-tuples $\vec{m}=(m_1,m_2,\dots, m_{k}, m_{k+1})$ satisfying $m_1\geq m_2\geq \dots \geq m_k\geq 0$, $m_k\geq m_{k+1}>-\infty$. Then
\begin{multline}\label{mseq103}
\sum_{\vec{m}}\frac{q^{k\left[m_1(m_1-m_2)+m_2(m_2-m_3)+\dots +m_{k}(m_{k}-m_{k+1})-m_1\right]+m_1+m_2+\dots+m_k}}
{\prod_{j=1}^{k}(q^k;q^k)_{m_j-m_{j+1}}\prod_{j=1}^k(q^j;q^k)_{m_j} }\\
\times (q^{p/2})^{ m_{k+1}^2}(-q^{p/2-i})^{m_{k+1}}
=\frac{(q^i,q^{p-i},q^p;q^p)_{\infty}}{(q;q)_{\infty}}.
\end{multline}
\end{corollary}

Perhaps not surprisingly, applications of the $k=1$ case of Theorem \ref{multsumt} are more common in the literature, so we consider this case  in more detail in a later section (actually the $k=1$ case was discovered first, before it was noticed that the process could be iterated to give Theorem \ref{multsumt} in its full generality). One example of an application of this $k=1$ case is the following  identity for the continuous $q$-ultraspherical polynomials, $ C_n(\cos \theta;\beta |q)$.
\begin{corollary}\label{qultra}
If $|c e^{i\theta}/(a\beta)|,\,|c e^{2i\theta}/(a\beta)|<1$, then
\begin{multline}\label{usp1}
\sum_{n=0}^{\infty}\frac{(a;q)_n}{(c;q)_n}\left(\frac{ce^{i\theta}}{a\beta}\right)^n C_n(\cos \theta;\beta |q)\\=
\frac{(c/a,c/\beta;q)_{\infty}}{(c,c/a\beta;q)_{\infty}}
\sum_{n=0}^{\infty}\frac{(a,\beta;q)_n}{(c/\beta,q;q)_n}\left(\frac{ce^{2i\theta}}{a\beta}\right)^n.
\end{multline}
\end{corollary}

The remainder of the paper proceeds as follows. We first prove two general transformations, each of which converts a double sum to a single sum, and then Theorem \ref{multsumt} is derived by iterating the result in one of these theorems. In the  section following that we consider some explicit applications of the $k=1$ case of Theorem \ref{multsumt}. Next, one of these transformations is recast as a Bailey-type transformation, and several applications of this are given. Finally, we pose a number of open questions.

 We employ the usual notations:
 {\allowdisplaybreaks
\begin{align*}
           (a;q)_n &:= (1-a)(1-aq)\cdots (1-aq^{n-1}), \\
          (a_1, a_2, \dots, a_j; q)_n &:= (a_1;q)_n (a_2;q)_n \cdots (a_j;q)_n ,\\
           (a;q)_\infty &:= (1-a)(1-aq)(1-aq^2)\cdots, \mbox{ and }\\
          (a_1, a_2, \dots, a_j; q)_\infty &:= (a_1;q)_\infty (a_2;q)_\infty \cdots (a_j;q)_\infty.
\end{align*}}

\section{Background and  Main Results}\label{sec:1}

In Pak's wonderful survey \cite{P06}, he asks (problem (2.3.2)) for a combinatorial proof
of the following identity (Pak's notation has been modified to the more usual $q$-series
notation):
\begin{equation}\label{eq1}
\sum_{m, n\geq
0}\frac{q^{m^2-mn+n^2}z^{m-n}}{(q;q)_m(q;q)_n}
=\frac{1}{(q;q)_{\infty}}\sum_{k=-\infty}^{\infty}z^kq^{k^2}.
\end{equation}
While searching for an \emph{analytic} proof of this identity, it became clear that a more general
identity was true, namely (assuming convergence),
\begin{equation*}
\sum_{m, n\geq 0}\frac{q^{mn}g(m-n)}{(q;q)_m(q;q)_n}
=\frac{1}{(q;q)_{\infty}}\sum_{k=-\infty}^{\infty}g(k).
\end{equation*}

In fact an even more general transformation holds.

\begin{theorem}\label{t2}
Let $g(k)$ be any function such that both series in \eqref{t2eq1}
converge. Then
\begin{multline}\label{t2eq1}
\sum_{m, n\geq 0}\frac{(a;q)_m(a;q)_n(q/a;q)_{m-n}}{(q;q)_m(q;q)_n(a;q)_{m-n}}\left(\frac{q}{a^2}\right)^{n}g(m-n)\\
=\frac{(q/a,q/a;q)_{\infty}}{(q,q/a^2;q)_{\infty}}
\sum_{k=-\infty}^{\infty}g(k).
\end{multline}
\end{theorem}

Before coming to the proof, we first recall the $q$-Gauss sum
\begin{equation}\label{qgs}
\sum_{n=0}^{\infty}\frac{(a,b;q)_n}{(c,q;q)_n}\left(\frac{c}{ab}\right)^n
=\frac{(c/a,c/b;q)_{\infty}}{(c,c/ab;q)_{\infty}}.
\end{equation}
\begin{proofof}{ Theorem \ref{t2}}
In \eqref{t2eq1}, set $m-n=k$ or $m=n+k$, so that the left
side becomes
\[
\sum_{k=-\infty}^{\infty}\frac{(q/a;q)_k}{(a;q)_k}g(k)\sum_{n}\frac{(a;q)_{n+k}(a;q)_n}
{(q;q)_{n+k}(q;q)_n}\left(\frac{q}{a^2}\right)^{n},
\]
where the sum on $n$ is over $n\geq 0$ if $k\geq 0$ and over $n\geq
-k$ if $k< 0$. If $k \geq 0$,
{\allowdisplaybreaks\begin{align*}
\sum_{n\geq 0}\frac{(a;q)_{n+k}(a;q)_n}
{(q;q)_{n+k}(q;q)_n}&\left(\frac{q}{a^2}\right)^{n}\\
&=\frac{(a;q)_k}{(q;q)_k}
\sum_{n\geq 0}\frac{(aq^k;q)_{n}(a;q)_n}
{(q^{k+1};q)_{n}(q;q)_n}\left(\frac{q}{a^2}\right)^{n}\\
&=\frac{(a;q)_k}{(q;q)_k}
\frac{(q/a,q^{k+1}/a;q)_{\infty}}{(q^{k+1},q/a^2;q)_{\infty}}\\
&=
\frac{(q/a,q/a;q)_{\infty}}{(q,q/a^2;q)_{\infty}}
\frac{(a;q)_k}{(q/a;q)_k},
\end{align*}}
by the $q$-Gauss sum \eqref{qgs} above (replace $a$ with $q^k a$,
$b$ with $a$, $c$ with $q^{k+1}$). A similar argument works when
$k<0$.
\end{proofof}

A second general double summation identity is contained in the
following theorem.

\begin{theorem}\label{t3}
Let $g(k)$ be any function such that both series in \eqref{t3eq1}
converge. Then
\begin{multline}\label{t3eq1}
\sum_{m\geq n\geq 0}\frac{(a;q)_m(b;q)_n(c/b;q)_{m-n}}{(c;q)_m(q;q)_n(a;q)_{m-n}}\left(\frac{c}{ab}\right)^{n}g(m-n)\\
=\frac{(c/a,c/b;q)_{\infty}}{(c,c/ab;q)_{\infty}}
\sum_{k=0}^{\infty}g(k).
\end{multline}
\end{theorem}

\begin{proof}
Set $m-n=k$ or $m=n+k$, so that the left side becomes
\[
\sum_{k=0}^{\infty}g(k)\frac{(c/b;q)_k}{(a;q)_k}\sum_{n\geq 0}\frac{(a;q)_{n+k}(b;q)_n}
{(c;q)_{n+k}(q;q)_n}\left(\frac{c}{ab}\right)^{n},
\]
and {\allowdisplaybreaks\begin{align*} \sum_{n\geq
0}\frac{(a;q)_{n+k}(b;q)_n}
{(c;q)_{n+k}(q;q)_n}\left(\frac{c}{ab}\right)^{n}
&=\frac{(a;q)_k}{(c;q)_k} \sum_{n\geq 0}\frac{(aq^k;q)_{n}(b;q)_n}
{(cq^{k};q)_{n}(q;q)_n}\left(\frac{c}{ab}\right)^{n}\\
&=\frac{(a;q)_k}{(c;q)_k}
\frac{(c/a,cq^{k}/b;q)_{\infty}}{(cq^{k},c/ab;q)_{\infty}}\\
&= \frac{(c/a,c/b;q)_{\infty}}{(c,c/ab;q)_{\infty}}
\frac{(a;q)_k}{(c/b;q)_k},
\end{align*}}
by the $q$-Gauss sum \eqref{qgs} above (replace $a$ with $aq^k$ and
$c$ with $cq^{k}$).
\end{proof}
Remark: There is obviously some overlap between Theorems \ref{t2} and \ref{t3}, but neither is contained in the other.

\subsection{Multi-sums and the Main Theorems}
By multi-sums we mean here nested multiple sums of arbitrary depth. See  Andrews' \cite{A74} analytic version of the Andrews-Gordon identities at \eqref{AGids}
in the introduction for an example, and also the references mentioned there for further examples.
The constructions in the present paper may be iterated to produce multi-sums of a somewhat similar nature. We next prove Theorem \ref{multsumt}.

\begin{proofof}{Theorem \ref{multsumt}}
Rewrite \eqref{t3eq1} (after replacing $a$ with $a_1$, $b$ with $b_1$ and $c$ with $c_1$, $m$ with $m_1$, $n$ and $k$ with $m_2$, and finally replacing $m_2$ on the left side with $m_1-m_2$) as
\begin{multline}\label{t3eq11}
\sum_{m_1\geq m_2\geq 0}\frac{(a_1;q)_{m_1}(c_1/b_1;q)_{m_2}(b_1;q)_{m_1-m_2}}{(c_1;q)_{m_1}(a_1;q)_{m_2}(q;q)_{m_1-m_2}}\left(\frac{c_1}{a_1b_1}\right)^{m_1-m_2}g(m_2)\\
=\frac{(c_1/a_1,c_1/b_1;q)_{\infty}}{(c_1,c_1/a_1b_1;q)_{\infty}}.
\sum_{m_2=0}^{\infty}g(m_2).
\end{multline}
This is the $k=1$ case of Theorem \ref{multsumt}. The $k=2$ case easily follows upon setting
\[
g(m_2) = \sum_{m_3=0}^{m_2}\frac{(a_2;q)_{m_2}(c_2/b_2;q)_{m_3}(b_2;q)_{m_2-m_3}}
{(c_2;q)_{m_2}(a_2;q)_{m_3}(q;q)_{m_2-m_3}}
\left(\frac{c_2}{a_2b_2}\right)^{m_2-m_3}g(m_3),
\]
and using \eqref{t3eq11} to sum the resulting right side. This process can be repeated to arbitrary depth, giving the theorem.
\end{proofof}

It is natural to ask if Theorem \ref{t2} can be similarly iterated. The answer is ``yes", once it is noticed that the sum on the left side of \eqref{t2eq1}
may be extended to $\sum_{m=-\infty}^{\infty}$ for free, since $1/(q;q)_m=0$ for $m<0$. However, a more general identity may be derived by modifying the proof of the previous theorem.
\begin{proofof}{Theorem \ref{multsumtt}}
The proof follows the proof of Theorem \ref{multsumt}, except at the last stage we instead set
\[
g(m_{k})=\sum_{m_{k+1}=-\infty}^{m_k}\frac{(a;q)_{m_k}(q/a;q)_{m_{k+1}}(a;q)_{m_k-m_{k+1}}}
{(q;q)_{m_k}(a;q)_{m_{k+1}}(q;q)_{m_k-m_{k+1}}}
\left(\frac{q}{a^2}\right)^{m_k-m_{k+1}}g(m_{k+1}),
\]
and then use \eqref{t2eq1} to sum the final right side.
\end{proofof}

Any sequence $\{g(j)\}_{j=0}^{\infty}$ which is summable to an infinite product may now be substituted in \eqref{multc}, to give a multi-sum equals infinite product identity. This includes all the sequences from any of the known basic hypergeometric summation formulae, and in particular any of the 130 identities on the Slater list. Likewise, any sequence $\{g(j)\}_{j=-\infty}^{\infty}$ which is summable to an infinite product may now be substituted in \eqref{multcc}, to also give a multi-sum equals infinite product identity.
\begin{corollary}
Let $|q|<1$, $k\geq 1$ be a positive integer,  $a_1, \dots , a_k$,  $b_1, \dots , b_k$,  $c_1, \dots , c_k$ be complex numbers with each $|c_j/(a_jb_j)|<1$. Let the sum on the left be over all integer  $k+1$-tuples $\vec{m}=(m_1,m_2,\dots, m_{k+1})$ satisfying $m_1\geq m_2\geq \dots \geq m_k\geq m_{k+1}\geq 0$. Then
{\allowdisplaybreaks
\begin{multline}\label{multcr}
\sum_{\vec{m}}\prod_{j=1}^{k}\frac{(a_j;q)_{m_j}(c_j/b_j;q)_{m_{j+1}}(b_j;q)_{m_j-m_{j+1}}}
{(c_j;q)_{m_j}(a_j;q)_{m_{j+1}}(q;q)_{m_j-m_{j+1}}}\left( \frac{c_j}{a_jb_j}\right)^{m_j-m_{j+1}} \frac{q^{m_{k+1}^2}}{(q;q)_{m_{k+1}}}\\
=\frac{1}{(q,q^4;q^5)_{\infty}}\prod_{j=1}^{k}\frac{(c_j/a_j,c_j/b_j;q)_{\infty}}{(c_j,c_j/a_jb_j;q)_{\infty}}.
\end{multline}
}
\end{corollary}
\begin{proof}
Set
\[
g(j) = \frac{q^{j^2}}{(q;q)_j}
\]
in Theorem \ref{multsumt}.
\end{proof}

\begin{corollary}
Let $k \geq 1$ be an integer. Assume that $\{g(j)\}$ is a sequence such that both sides following converge, and let the sum on the left be over all integer  $k+1$-tuples $\vec{m}=(m_1,m_2,\dots, m_{k+1})$ satisfying $m_1\geq m_2\geq \dots \geq m_k\geq m_{k+1}\geq 0$. Then
\begin{multline}\label{mseq1}
\sum_{\vec{m}}\frac{q^{m_1(m_1-m_2)+m_2(m_2-m_3)+\dots +m_k(m_k-m_{k+1})+m_{k+1}-m_1}}
{(c_1;q)_{m_1}(q;q)_{m_1-m_{2}} (c_2;q)_{m_2}(q;q)_{m_2-m_{3}}  \dots (c_k;q)_{m_k}(q;q)_{m_k-m_{k+1}} }\\
\times
g(m_{k+1})\prod_{j=1}^kc_j^{m_j-m_{j+1}}
=\frac{1}{(c_1,c_2,\dots,c_k;q)_{\infty}}\sum_{j=0}^{\infty}g(j).
\end{multline}
\end{corollary}

\begin{proof}
Let each $a_j$, $b_j\to \infty$ in \eqref{multc}. (Alternatively, apply an argument similar to that used in the proof of Theorem \ref{multsumt} to iterate \eqref{eq2}, after first defining $g(j)=0$ for $j<0$).
\end{proof}

Corollary  \ref{multeuler} follows as a special case.
\begin{proofof}{ Corollary \ref{multeuler}}
For \eqref{mseq10}, let each $c_j=q$ and set $g(j)=\delta_{0,j}$ in the corollary above. For \eqref{mseq102}, replace $q$ with $q^k$, let $c_j=q^j$ and again set $g(j)=\delta_{0,j}$ in the corollary above.
\end{proofof}

\begin{corollary}
Let $k \geq 1$ be an integer. Assume that $\{g(j)\}$ is a sequence such that both sides following converge, and let the sum on the left be over all integer  $k+1$-tuples $\vec{m}=(m_1,m_2,\dots, m_{k+1})$ satisfying $m_1\geq m_2\geq \dots \geq m_k\geq 0$, $m_k  \geq m_{k+1}>-\infty$. Then
\begin{multline}\label{mseq111}
\!\sum_{\vec{m}}\frac{q^{m_1(m_1-m_2)+m_2(m_2-m_3)+\dots +m_k(m_k-m_{k+1})+m_{k}-m_1}}
{(c_1;q)_{m_1}(q;q)_{m_1-m_{2}} (c_2;q)_{m_2}(q;q)_{m_2-m_{3}}  \dots (c_{k-1};q)_{m_{k-1}}(q;q)_{m_{k-1}-m_{k}}}\\
\times
\frac{g(m_{k+1})\prod_{j=1}^{k-1}c_j^{m_j-m_{j+1}}}{(q;q)_{m_{k}}(q;q)_{m_{k}-m_{k+1}}}
=\frac{1}{(c_1,c_2,\dots,c_{k-1},q;q)_{\infty}}\sum_{j=-\infty}^{\infty}g(j).
\end{multline}
\end{corollary}

\begin{proof}
Let $a$ and each $a_j$, $b_j\to \infty$ in \eqref{multcc}.
\end{proof}

Corollary  \ref{corbilat} follows as a special case.
\begin{proofof}{ Corollary \ref{corbilat}}
In the corollary above, replace $q$ with $q^k$, set each $c_j=q^j$ and set $g(j)=q^{pj^2/2}(-q^{p/2-i})^j$, and simplify.
\end{proofof}

\section{Some Applications}

We first consider a special case of Theorem \ref{t2} which has a number of interesting implications.

\begin{corollary}\label{tc1}
Let $g(k)$ be any function such that both series in \eqref{eq2} converge. Then
\begin{equation}\label{eq2}
\sum_{m, n\geq 0}\frac{q^{mn}g(m-n)}{(q;q)_m(q;q)_n}=\frac{1}
{(q;q)_{\infty}}\sum_{k=-\infty}^{\infty}g(k).
\end{equation}
\end{corollary}

\begin{proof}
Let $a\to \infty$ in \eqref{t2eq1}, and \eqref{eq2} follows after some simple algebra.
\end{proof}

We first give another demonstration that the Jacobi triple product identity follows from the following special case of the $q$-binomial theorem:
\begin{equation}\label{cauchy}
\sum_{n=0}^{\infty}\frac{q^{n(n+1)/2}x^n}{(q;q)_n}=(-xq;q)_{\infty}.
\end{equation}
\begin{corollary}\label{c1b}
Let $z$ be a non-zero complex number. If $|q|<1$, then
\begin{equation}\label{jttid}
\sum_{n=
-\infty}^{\infty}q^{n^2}z^n=(-qz,-q/z,q^2;q^2)_{\infty}.
\end{equation}
\end{corollary}
\begin{proof}
In \eqref{eq2}, set \[
g(i)=q^{i^2/2}z^i,
\]
so that this identity becomes
\[
\sum_{m\geq0}\frac{q^{m^2/2}z^m}{(q;q)_m}\sum_{n\geq0}\frac{q^{n^2/2}z^{-n}}{(q;q)_n}=\frac{1}{(q;q)_{\infty}}
\sum_{k=-\infty}^{\infty}q^{k^2/2}z^k.
\]
Now apply \eqref{cauchy} to the two sums on the left side (with $x$ replaced with $z/q^{1/2}$ and $1/(zq^{1/2})$), replace $q$ with $q^2$, and \eqref{jttid} follows.
\end{proof}
Remark: Andrews \cite{A65} gave a different proof the Jacobi triple product identity follows from the $q$-binomial theorem. The identity at  \eqref{eq1} also now follows as a special case of Corollary \ref{tc1}.

\begin{corollary}\label{c1}
If $|q|<1$ and $z\not =0$, then
\begin{equation*}
\sum_{m, n\geq
0}\frac{q^{m^2-mn+n^2}z^{m-n}}{(q;q)_m(q;q)_n}=\frac{(-q/z,-qz;q^2)_{\infty}}{(q;q^2)_{\infty}}.
\end{equation*}
\end{corollary}
\begin{proof}
Set $g(i)=q^{i^2}z^i$ in \eqref{eq2} and use the Jacobi triple product identity \eqref{jttid} above.
\end{proof}
Remark: The case $z=1$ gives an identity proved by Andrews in \cite{A77}.  In the same paper \cite{A77}, this identity motivated Andrews to pose the question:
``For what positive definite quadratic forms $Q(m,n)$  is
\[
\sum_{m=0}^{\infty}\sum_{n=0}^{\infty}\frac{q^{Q(m,n)}}{(q;q)_m(q;q)_n}
\]
summable to an infinite product. He also remarks that ``The only non-diagonal forms I know of are $km^2
+kn^2-(2k-1)mn$ ($k$ positive integral) and $n^2+2m^2+2nm$.'' The result for this infinite family of $k$-values also follows easily from Corollary \ref{tc1}.

\begin{corollary}
If $|q|<1$ and $k\geq 1$ is integral, then
\begin{equation}\label{mnkqid}
\sum_{m, n\geq
0}\frac{q^{km^2-(2k-1)mn+kn^2}}{(q;q)_m(q;q)_n}=\frac{(-q^k,-q^k,q^{2k};q^{2k})_{\infty}}{(q;q)_{\infty}}.
\end{equation}
\end{corollary}
\begin{proof}
Set $g(i)=q^{ki^2}$ in \eqref{eq2} and use the Jacobi triple product identity \eqref{jttid} above.
\end{proof}
Remark: The above identity was also proved by Andrews in \cite{A81} (Equation (4.2)).

While identities of the form ``infinite double-sum = infinite product" are possibly not quite so interesting as ``infinite single sum = infinite product" identities of the Rogers-Ramanujan-Slater, they are of some interest, and do appear in the literature. There are no known single-sum identities in which the modulus in the infinite product  is 11, but there double-sum identities of this type, stated in \cite{A75} by Andrews. Another example was given by Andrews in \cite{A77},  where a double-sum  alternative to one of the mod 7 identities due to Rogers was given:
\begin{equation}\label{amod7id}
\sum_{m, n\geq
0}\frac{q^{2m^2+2mn+n^2}}{(q;q)_m(q;q)_n}=\frac{(q^3,q^4,q^{7};q^{7})_{\infty}}{(q;q)_{\infty}}.
\end{equation}

It is clear that Corollary \ref{tc1} will also give many other double series that may be expressed as infinite products.
\begin{corollary}\label{c33}
If $|q|<1$, and  $k\geq 1$ and $0\leq j<k$ are integers with $j+k$ even, then
\begin{equation}
\sum_{m, n\geq
0}\frac{q^{(km^2-(2k-2)mn+n^2+j(m-n))/2}(-1)^{m-n}}{(q;q)_m(q;q)_n}=\frac{(q^{(k-j)/2},q^{(k+j)/2},q^{k};q^{k})_{\infty}}{(q;q)_{\infty}}.
\end{equation}
\end{corollary}
\begin{proof}
Set $g(i)=(-1)^iq^{(ki^2+ji)/2}$ in Corollary \ref{tc1} and once again use the Jacobi triple product identity \eqref{jttid} to sum the right side.
\end{proof}

For example, setting $k=7$ and $j=1$ in  Corollary \ref{c33} gives a double-sum identity with the same product side as that of Andrews at \eqref{amod7id}:
\[
\sum_{m, n\geq
0}\frac{q^{(7m^2-12mn+7n^2+m-n)/2}(-1)^{m-n}}{(q;q)_m(q;q)_n}=\frac{(q^3,q^4,q^{7};q^{7})_{\infty}}{(q;q)_{\infty}}.
\]

Letting $g(i)$ be the $i$-th term in the series side of any Rogers-Ramanujan-Slater-type identity (including the 130 such identities on the Slater list) will also lead to a double  summation formula.
\begin{corollary}\label{c3}
If $|q|<1$ then
\begin{equation}\label{c3eq1}
\sum_{m\geq n\geq 0}\frac{(a;q)_m(a;q)_n(q/a;q)_{m-n}}
{(q;q)_m(q;q)_n(a,q;q)_{m-n}}\frac{q^{m^2-2mn+n^2+n}}{a^{2n}}
=\frac{(q/a,q/a;q)_{\infty}}
{(q/a^2,q;q)_{\infty}(q,q^4;q^5)_{\infty}}.
\end{equation}
\end{corollary}
\begin{proof}
Set\[
g(i)=\frac{q^{i^2}}
{(q;q)_{i}}
\] for $i\geq 0$, and equal to 0 for $i<0$, in \eqref{t2eq1}, and use the first Rogers-Ramanujan identity:
\begin{equation}\label{rr1}
\sum_{k=0}^{\infty}\frac{q^{i^2}}
{(q;q)_{i}}
=\frac{ 1 } {
(q,q^4;q^5)_{\infty} }.
\end{equation}
\end{proof}

\emph{Any} (uni-lateral or bi-lateral) basic hypergeometric  summation formula may be used in \eqref{t2eq1} to produce a double-summation identity (simply let $g(k)$ be the $k$-th term in the basic hypergeometric  sum).
Indeed, it is not necessary that the sequence $\{g(i)\}$ be basic hypergeometric in nature. The following amusing result is also a consequence of Theorem \ref{t2}.
\begin{corollary}\label{c5}
If $|q/a^2|<1$, then
\begin{equation*}
\sum_{m> n\geq 0}\frac{(a;q)_m(a;q)_n(q/a;q)_{m-n}q^{n}}
{(q;q)_m(q;q)_n(a,q;q)_{m-n}a^{2n}(m-n)^2}
=\frac{\pi^2(q/a,q/a;q)_{\infty}}
{6(q/a^2,q;q)_{\infty}}.
\end{equation*}
\end{corollary}
\begin{proof}
Define \[
g(i)=\begin{cases}
\frac{1}{i^2},&i>0,\\
0, & \text{otherwise}
\end{cases}\] in \eqref{eq2} and use the fact that $\zeta(2)=\pi^2/6$.
\end{proof}

As with Theorem \ref{t2}, Theorem \ref{t3} may also be employed in conjunction with existing
summation formulae to produce double summation identities. We give one example.

\begin{corollary}\label{t3c4}
Let $A$, $B$, $C$, $a$,  $c$ and $d$  be  such that  none of the
denominators below vanish, with $|q|, |c|, |C/AB|<1$ and . Then
\begin{multline}\label{t3c4eq1}
\sum_{m\geq n\geq 0}\!\!\!
\frac{\left(-c,q\sqrt{-c},-q\sqrt{-c},a,\frac{q}{a},c,-d,\frac{-q}{d},\frac{C}{B};q\right)_{m-n}(A;q)_m(B;q)_nc^{m-n}C^n}
{(\sqrt{-c},-\sqrt{-c},\frac{-cq}{a},-ac,-q,\frac{cq}{d},cd,q,A;q)_{m-n}(C;q)_m(q;q)_nA^{n}B^n}
\\
=\frac{(C/A,C/B,-c,-cq;q)_{\infty}(acd,acq/d,cdq/a,cq^2/ad;q^2)_{\infty}}
{(C/AB,C,cd,cq/d,-ac,-cq/a;q)_{\infty}}.
\end{multline}
\end{corollary}

\begin{proof}
Replace $a$ with $A$, $b$ with $B$, $c$ with $C$ and set \[
g(i)=\begin{cases}
\frac{(-c,q\sqrt{-c},-q\sqrt{-c},a,q/a,c,-d,-q/d;q)_{i}c^{i}}
{(\sqrt{-c},-\sqrt{-c},-cq/a,-ac,-q,cq/d,cd;q)_{i}},&i\geq 0,\\
0, & \text{otherwise}
\end{cases}\] in \eqref{t3eq1} and  use the $q$-analogue of Whipple's $_3F_2$ sum \eqref{wq3f2}\begin{multline}\label{wq3f2}
\sum_{k=0}^{\infty}\frac{  (
-c, \,q \sqrt{-c},\, -q \sqrt{-c}, \,a,\, q/a,\, c,\,-d, \,
-q/d;q)_kc^k}{
(\sqrt{-c},\,-\sqrt{-c},\, -c q/a,\,-a c,\,-q,\,c q/d,\, c d,q;q)_k}
 \\ = \frac{(-c, -cq;q)_{\infty} (a c d, a c q/d, c d
q/a,c q^2/a d;q^2)_{\infty}}{( cd, c q/d, -a c, -cq/a;q)_{\infty}}.
\end{multline}
to sum the right side.
\end{proof}

\section{A Bailey-type Transform}\label{sbtf}

Theorem \ref{t3} above may be recast as a  transformation involving restricted WP-Bailey pairs. As will be seen below,
one reason for doing this is that the resulting transformation appears to hint at an (as of now) undiscovered quite general WP-Bailey chain.
For comparison purposes (the reason to be outlined below), we recall Andrews' \cite{A01} definition of a \emph{WP-Bailey
pair},  namely a pair of sequences $(\alpha_{n}(a,k,q)$,
$\beta_{n}(a,k,q))$  satisfying $\alpha_{0}(a,k,q)$
$=\beta_{0}(a,k,q)$ and {\allowdisplaybreaks
\begin{align}\label{WPpair}
\beta_{n}(a,k,q) &= \sum_{j=0}^{n}
\frac{(k/a;q)_{n-j}(k;q)_{n+j}}{(q;q)_{n-j}(aq;q)_{n+j}}\alpha_{j}(a,k,q).
\end{align}
}

A limiting case of Andrews' first WP-Bailey chain gives that
 if $(\alpha_n, \beta_n)$ satisfy \eqref{WPpair}, then
subject to suitable convergence conditions, {\allowdisplaybreaks
\begin{multline}\label{wpeq}
\sum_{n=0}^{\infty} \frac{(q\sqrt{k},-q\sqrt{k}, y,z;q)_{n}}
{(\sqrt{k},-\sqrt{k}, q k/y,q k/z;q)_{n}}\left( \frac{q a}{y z }\right )^{n} \beta_n =\\
\frac{(q k,q k/yz,q a/y,q a/z;q)_{\infty}} {(q k/y,q k/z,q a,q
a/yz;q)_{\infty}} \sum_{n=0}^{\infty}\frac{(y,z;q)_{n}}{(q a/y ,q
a/z;q)_n}\left (\frac{q a}{y z}\right)^{n}\alpha_n.
\end{multline}
}

We now prove the  Bailey-type transformation alluded to in the title of this section.

\begin{theorem}\label{bt2}
If
\begin{equation}\label{albet2}
\beta_m=\sum_{n=0}^m\frac{(b;q)_{m-n}}{(q;q)_{m-n}}\alpha_n,
\end{equation}
then
\begin{equation}\label{btrans2}
\sum_{m=0}^{\infty}\frac{(a;q)_m}{(c;q)_m}\left(\frac{c}{ab} \right)^m\beta_m
=\frac{(c/a,c/b;q)_{\infty}}{(c,c/ab;q)_{\infty}}\sum_{k=0}^{\infty}\frac{(a;q)_k}{(c/b;q)_k}
\left(\frac{c}{ab} \right)^k\alpha_k.
\end{equation}
\end{theorem}

\begin{proof}
Replace $g(i)$ with $\alpha_i$ in Theorem \ref{t3}, so that
{\allowdisplaybreaks
\begin{align}\label{bt2eq2}
\frac{(c/a,c/b;q)_{\infty}}{(c,c/ab;q)_{\infty}}
\sum_{k=0}^{\infty}&\alpha_k=\sum_{m\geq n\geq 0}\frac{(a;q)_m(b;q)_n(c/b;q)_{m-n}}{(c;q)_m(q;q)_n(a;q)_{m-n}}\left(\frac{c}{ab}\right)^{n}\alpha_{m-n}\\
&=\sum_{m=0}^{\infty} \frac{(a;q)_m}{(c;q)_m} \sum_{n=0}^m\frac{(b;q)_n(c/b;q)_{m-n}}{(q;q)_n(a;q)_{m-n}}\left(\frac{c}{ab}\right)^{n}\alpha_{m-n}\notag\\
&=\sum_{m=0}^{\infty} \frac{(a;q)_m}{(c;q)_m} \left(\frac{c}{ab}\right)^{m} \sum_{n=0}^m\frac{(b;q)_{m-n}(c/b;q)_{n}}{(q;q)_{m-n}(a;q)_{n}}\left(\frac{c}{ab}\right)^{-n}\alpha_{n}.\notag
\end{align}}
Now make the replacement
\[
\alpha_k \to \frac{(a;q)_k}{(c/b;q)_k}
\left(\frac{c}{ab} \right)^k\alpha_k
\]
and the result follows.
\end{proof}
Remarks: 1) It is clear that replacing $k$ with $ak$, letting $a\to 0$ and then setting $k=b$ in \eqref{WPpair} gives a pair defined by  \eqref{albet2}. However, it does not appear that \eqref{btrans2} follows upon making the same substitutions in any of the existing WP-Bailey chains. Indeed, the only such chain containing free parameters different from $a$ and $k$ (the transformation \eqref{btrans2} has three free parameters $a$, $b$ and $c$) is Andrews first WP-Bailey chain, and it is not difficult to see that replacing $k$ with $ak$, letting $a\to 0$ and then setting $k=b$ in this chain results in a trivial identity. It may be that \eqref{btrans2}  follows from some as yet undiscovered WP-Bailey chain.

2) If Theorem \ref{t2} is recast as a Bailey-type transform, the result is merely in a special case of Theorem \ref{bt2}.

As remarked above, it may be that the transformation at \eqref{btrans2} above may be a restricted version of a a full (as yet unknown) WP-Bailey chain, so possibly its main interest at present is possibly as an indicator of this chain. As it stands  (one might say it is only a ``shadow'' of the full WP-Bailey chain that it possibly hints at), the identities resulting from substituting pairs deriving from existing WP-Bailey pairs for the most part lead to known identities.

\subsection{Two companions to an identity of Andrews}

One implication we believe to be new is a pair of companion identities to a result \cite[Theorem 7]{A66} of Andrews.
\begin{corollary}\label{corbaprs}
If $|q|,|c/aq|<1$, then
\begin{multline}\label{btrans2eq2}
\sum_{m=0}^{\infty}\frac{(a,1/b;q^2)_m}{(c,q^2;q^2)_m}\left(\frac{c}{aq} \right)^m\\
=\frac{(c/a,c/b;q^2)_{\infty}}{(c,c/ab;q^2)_{\infty}}\sum_{k=0}^{\infty}
\frac{(a;q^2)_k(1/b;q)_k}{(c/b;q^2)_k(q;q)_k}
\left(\frac{c}{aq} \right)^k;
\end{multline}
\begin{multline}\label{btrans2eq22}
\sum_{m=0}^{\infty}\frac{(a,q^2/b;q^2)_m}{(c,q^2;q^2)_m}\left(\frac{c}{aq} \right)^m\\
=\frac{(c/a,c/b;q^2)_{\infty}}{(c,c/ab;q^2)_{\infty}}\sum_{k=0}^{\infty}
\frac{(a;q^2)_k(q/b;q)_k}{(c/b;q^2)_k(q;q)_k}
\left(\frac{c}{aq} \right)^k.
\end{multline}
\end{corollary}
\begin{proof}
Start with the WP-Bailey pair of Bressoud \cite{B81}
\begin{align}\label{Bpr3}
\alpha_n(a,k)&=\frac{1-a\,q^{2n}}{1-a}\,
\frac{\left(\sqrt{a},\frac{a}{k};\sqrt{q}\right)_n}
{\left(\sqrt{q},k\sqrt{\frac{q}{ a}};\sqrt{q}\right)_n}
\left(\frac{k}{a \sqrt{q}} \right)^n,\\
\beta_n(a,k)&=\frac{\left(k,\frac{a}{k},-k\sqrt{\frac{q}{a}},-\frac{k
q}{\sqrt{a}};q\right)_n}{\left(q, \frac{q k^2}{a},-\sqrt{a},-\sqrt{a
q};q\right)_n} \left(\frac{k}{a \sqrt{q}} \right)^n,\notag
\end{align}
and making the same substitutions listed above (replacing $k$ with $ak$, letting $a\to 0$ and then setting $k=b$) leads to the  pair
\begin{align}
\alpha_n&=\frac{(1/b;\sqrt{q})_n}{(\sqrt{q};\sqrt{q})_n}\left(\frac{b}{\sqrt{q}}\right)^n,\\
\beta_n&=\frac{(1/b;q)_n}{(q;q)_n}\left(\frac{b}{\sqrt{q}}\right)^n. \notag
\end{align}
Substitution of this latter pair into \eqref{btrans2}, and then replacing $\sqrt{q}$ with $q$ leads to the identity at \eqref{btrans2eq2} above. Applying the same treatment to a second WP-Bailey pair due to Bressoud \cite{B81}
\begin{align}\label{Bpr2}
\alpha_n(a,k)&=\frac{1-\sqrt{a}\,q^n}{1-\sqrt{a}}\,
\frac{\left(\sqrt{a},\frac{a \sqrt{q}}{k};\sqrt{q}\right)_n}
{\left(\sqrt{q},\frac{k}{ \sqrt{a}};\sqrt{q}\right)_n}
\left(\frac{k}{a \sqrt{q}} \right)^n,\\
\beta_n(a,k)&=\frac{\left(k,\frac{a q}{k};q\right)_n}{\left(q,
\frac{k^2}{a};q\right)_n}
\frac{\left(\frac{-k}{\sqrt{a}};\sqrt{q}\right)_{2n}}{\left( -
\sqrt{a q};\sqrt{q}\right)_{2n}}\left(\frac{k}{a \sqrt{q}}
\right)^n, \notag
\end{align}
gives \eqref{btrans2eq22} above.
\end{proof}

This identity at \eqref{btrans2eq2} above is easily seen to be equivalent to the identity
\[
\sum_{m=0}^{\infty}\frac{(a,b;q^2)_m}{(bt,q^2;q^2)_m}\left(\frac{t}{q} \right)^m\\
=\frac{(abt,t;q^2)_{\infty}}{(bt,at;q^2)_{\infty}}\sum_{k=0}^{\infty}
\frac{(b;q^2)_k(a;q)_k}{(abt;q^2)_k(q;q)_k}
\left(\frac{t}{q} \right)^k,
\]
while that at \eqref{btrans2eq22} is equivalent to the identity
\[
\sum_{m=0}^{\infty}\frac{(a,b;q^2)_m}{(bt,q^2;q^2)_m}\left(\frac{t}{q} \right)^m\\
=\frac{(abt/q^2,t;q^2)_{\infty}}{(bt,at/q^2;q^2)_{\infty}}
\sum_{k=0}^{\infty}
\frac{(b;q^2)_k(a/q;q)_k}{(abt/q^2;q^2)_k(q;q)_k}
\left(\frac{t}{q} \right)^k.
 \]
Both of these may be viewed as  companions to the afore-mentioned identity \cite[Theorem 7]{A66} of Andrews:
\[
\sum_{m=0}^{\infty}\frac{(a,b;q^2)_m}{(bt,q^2;q^2)_m}\left(tq\right)^m\\
=\frac{(abt,t;q^2)_{\infty}}{(bt,at;q^2)_{\infty}}\sum_{k=0}^{\infty}
\frac{(b;q^2)_k(a;q)_k}{(abt;q^2)_k(q;q)_k}
t^k.
\]
Remark: An identity equivalent to that of Andrews above may be derived  by treating the WP-Bailey pair
\begin{align}\label{cn333pr}
\alpha_n(a,k,q)&=\frac{1-\sqrt{a}q^n}{1-\sqrt{a}}\frac{\left(\sqrt{a},\frac{a}{k};\sqrt{q}
\right)_n}{\left(\sqrt{q},k\sqrt{\frac{q}{a}};\sqrt{q}
\right)_n}\left( \frac{k}{a}\right)^n,
\\
\beta_n(a,k,q)&=\frac{\left(
-\frac{k}{\sqrt{a}};\sqrt{q}\right)_{2n}}{\left( -\sqrt{a
q};\sqrt{q}\right)_{2n}}\frac{\left(
\frac{a}{k},k;q\right)_{n}}{\left( \frac{k^2
q}{a},q;q\right)_{n}}\left( \frac{k \sqrt{q}}{a}\right)^n, \notag
\end{align}
from \cite{MZ10} in the same manner as were the pairs of Bressoud in Corollary \ref{corbaprs} above.

\subsection{Identities involving orthogonal polynomials}

Another  (possibly new)  application of this transform is a transformation formula  for a series involving the continuous $q$-ultraspherical polynomials. These polynomials (see for example, \cite[page 527]{AAR99}) may be defined by
\begin{equation}
C_n(\cos \theta;\beta |q)=\sum_{k=0}^n\frac{(\beta;q)_k(\beta;q)_{n-k}}{(q;q)_k(q;q)_{n-k}}e^{i(n-2k)\theta}
\end{equation}

\begin{proofof}{Corollary \ref{qultra}}
Upon noting that
\[
C_n(\cos \theta;\beta |q)
=e^{-i n \theta}\sum_{k=0}^n\frac{(\beta;q)_k(\beta;q)_{n-k}}{(q;q)_k(q;q)_{n-k}}e^{2 i k \theta},
\]
replace $b$ with $\beta$ in Theorem \ref{bt2}, set
\[
\alpha_n = \frac{(\beta;q)_n}{(q;q)_n}e^{2in\theta},
\]
so that $\beta_n = e^{i n \theta}C_n(\cos (\theta);\beta |q)$, and \eqref{usp1} follows directly from \eqref{btrans2}, after substituting for $\alpha_n$ and $\beta_n$.
\end{proofof}

Another implication is a summation formula for a series involving the Al-Salam-Chihara polynomials, which may be defined  (see \cite[Page 381, Equation (15.1.12)]{I09}) as follows:
\begin{equation}\label{ascpolys}
p_n(\cos \theta;t_1,t_2 | q)=\frac{(q;q)_nt_1^n}{(t_1t_2;q)_n}
\sum_{k=0}^n\frac{(t_2e^{i\theta};q)_k(t_1e^{-i\theta};q)_{n-k}}
{(q;q)_k(q;q)_{n-k}}e^{i(n-2k)\theta}.
\end{equation}

\begin{corollary}
Let $p_n(\cos \theta;t_1,t_2 | q)$ be as at \eqref{ascpolys}, and suppose $|t_2 e^{-i\theta}/q|$, $|t_2 e^{i\theta}/q|$, $|q|<1$.
Then
\begin{equation}\label{pneq}
\sum_{n=0}^{\infty}p_n(\cos \theta;t_1,t_2 | q)\left( \frac{t_2}{t_1 q}\right)^n=
\frac{1-t_1t_2/q}{1-2t_2\cos(\theta)/q+t_2^2/q^2}.
\end{equation}
\end{corollary}

\begin{proof}
In Theorem \ref{bt2}, set $a=q$, $b=t_1 e^{-i\theta}$, $c=t_1t_2$ and
\[
\alpha_k = \frac{(t_2 e^{i\theta};q)_k}{(q;q)_k}e^{-2i\theta k}
\]
With these substitutions, the left side of \eqref{btrans2} becomes the left side of \eqref{pneq}, and after some simplification, the right side of \eqref{btrans2} becomes the right side of \eqref{pneq}, giving the result.
\end{proof}

\section{Concluding Remarks}

A number of questions may be asked.

1) The transformations in the present paper derive ultimately from the $q$-Gauss sum, and those of Andrews \cite{A77} and Chu \cite{Ch02,Ch05} derive ultimately from the $q$-Pfaff-Saalsch\"utz sum. Are there similar multi-sum-to-single-sum transformations that derive from other known summation formulae?

2) The transformation in Theorem \ref{bt2} may be re-cast as follows:
if
\begin{equation*}
\beta_m=\sum_{n=0}^m\frac{(k;q)_{m-n}}{(q;q)_{m-n}}\alpha_n,
\end{equation*}
then
\begin{equation*}
\sum_{m=0}^{\infty}\frac{(d;q)_m}{(c;q)_m}\left(\frac{c}{dk} \right)^m\beta_m
=\frac{(c/d,c/k;q)_{\infty}}{(c,c/dk;q)_{\infty}}\sum_{n=0}^{\infty}\frac{(d;q)_n}{(c/k;q)_n}
\left(\frac{c}{dk} \right)^n\alpha_n.
\end{equation*}
Does this transformation  derive from some as yet undiscovered WP-Bailey chain, after replacing $k$ with $ka$ and letting $a\to 0$?

3) Are there combinatorial proofs of the identities at \eqref{mseq10}, \eqref{mseq102} and \eqref{mseq103} above?

\begin{acknowledgements}
This work was partially supported by a grant from the Simons Foundation (\#209175 to James Mc Laughlin).
\end{acknowledgements}

% BibTeX users please use one of
%\bibliographystyle{spbasic}      % basic style, author-year citations
%\bibliographystyle{spmpsci}      % mathematics and physical sciences
%\bibliographystyle{spphys}       % APS-like style for physics
%\bibliography{}   % name your BibTeX data base

% Non-BibTeX users please use

\end{document}